\documentclass[12pt]{proc-l}
\usepackage[below]{placeins}

\usepackage{hyperref}
\usepackage[all]{hypcap}
\usepackage{amsmath,amscd,amssymb, epsfig,psfrag,subfigure}

\hypersetup{
pdfpagemode=FullScreen,
colorlinks=true}

\newtheorem{theorem}{Theorem}[section]
\newtheorem{lemma}[theorem]{Lemma}

\theoremstyle{definition}
\newtheorem{definition}[theorem]{Definition}

\newtheorem{conjecture}[theorem]{Conjecture}

\newcommand{\RR}{\mathbb{R}}

\def\vol{\mbox{\rm{Vol}}}

\def\d{\partial}

\def\Vol{\mbox{\rm{Vol}}}

\def\HH{\mathbb{H}}

\def\PP{\mathbb{P}}

\def\Vol{\mbox{\rm{Vol}}}

\def\HH{\mathbb{H}}

\def\PP{\mathbb{P}}

\def\split{\backslash\backslash}

\begin{document}

\title{The minimal volume orientable hyperbolic 2-cusped 3-manifolds} 

\author[Ian Agol]{%
        Ian Agol} 
\address{%
    University of California, Berkeley \\
    970 Evans Hall \#3840 \\
    Berkeley, CA 94720-3840} 
\email{%
     ianagol@math.berkeley.edu}  
\thanks{Agol partially supported by NSF grant DMS-0504975, and the Guggenheim foundation}

\date{%
 January 6, 2010}


\begin{abstract} 
We prove that the Whitehead link complement and the $(-2,3,8)$ pretzel link complement are the minimal volume orientable
hyperbolic 3-manifolds with two cusps, with volume $3.66...$ = 4 $\times$ Catalan's constant. We use topological arguments to establish
the existence of an essential surface which provides a lower bound on volume and strong
constraints on the manifolds that realize that lower bound.
\end{abstract} 

\maketitle
\section{Introduction}
Jorgensen and Thurston proved that the volumes of hyperbolic 3-manifolds are well-ordered. 
Moreover, if a volume is an $n$-fold limit point of smaller volumes (of order type $\omega^n$),
then there is a corresponding hyperbolic manifold of finite volume with precisely $n$ orientable
cusps.  The smallest volume one-cusped orientable
hyperbolic manifolds were identified by Cao and Meyerhoff \cite{CM}, one of which is the figure-eight knot complement,
with volume $2.0298\ldots= 2V_3$, where $V_3=1.01494\ldots$ is the volume of a regular ideal tetrahedron. 
Recently, Gabai, Meyerhoff and Milley have identified the smallest volume orientable hyperbolic 3-manifold
to be the Fomenko-Matveev-Weeks manifold, with volume $0.9427\ldots$ \cite{GMM07}. 
Moreover, they identify the 10 one-cusped orientable manifolds with volume $< 2.848$ 
(having five distinct volumes). It is an interesting question to find the smallest
volume manifolds with $n$ orientable cusps. Adams showed that
an $n$-cusped hyperbolic manifold has volume $\geq n V_3$ \cite{Adams88}. For orientable manifolds with two cusps, this was improved by Yoshida to a lower bound
of $2.43$ in \cite{Yoshida01}. 
In this paper, we prove that the Whitehead link 
complement $\mathbb{W}$ and the $(-2,3,8)$ pretzel link complement $\mathbb{W}'$ are the minimal volume orientable
hyperbolic 3-manifolds with two cusps, with volume $V_8=3.66...$ = 4 $\times$ Catalan's constant. These have been the smallest volume
known 2-cusped orientable hyperbolic 3-manifolds for quite some time, and we take it as given that these are known to have this volume (see \cite{Snappea} or \cite[p. 474]{Ratcliffe}). In fact, in Theorem \ref{surface volume}, we prove that if $M$ is a hyperbolic manifold
with a cusp and an essential surface disjoint from the cusp, then $\Vol(M)\geq V_8$, from which the result for 2-cusped manifolds follows in Theorem \ref{two cusps} by a result of Culler and Shalen \cite{CullerShalen84}. 

We now briefly outline the argument. Let $M$ be a compact orientable manifold with at least one boundary component, whose interior
$int(M)$ admits a finite volume hyperbolic metric and which contains a closed essential surface $X$ disjoint from the boundary. 
If the complement of 
$X$ has trivial ``\hyperlink{surfaceguts}{guts}'' (that is, the double  $D(M\split X)$  of $M$ split along $X$ is a graph manifold), then we use the JSJ decomposition of the complement of $X$ to show that the boundary
may be used to ``cut'' up the surface $X$, repeating until we get a surface which has non-trivial \hyperlink{surfaceguts}{guts}. Essentially what we are doing is replacing the initial surface $X$ which may have accidental parabolics, with a surface which has no accidental parabolics, but we must use a topological approach rather than geometric in order to keep track of the JSJ decomposition and get an acylindrical part of the \hyperlink{surfaceguts}{guts} in the end. 
We then apply
a volume estimate of \cite{AST05} and volume estimates
of Miyamoto \cite{Mi} giving lower bounds on volumes of hyperbolic manifolds with totally geodesic boundary, to get the 
volume lower bound on $int(M)$ (Theorem \ref{surface volume}). A recent sharp estimate \cite{CFW08} implies that we may characterize completely
the case of equality and identify the two smallest volume manifolds with two cusps by a simple combinatorial analysis (Theorem \ref{two cusps}). The
arguments of \cite{AST05, CFW08} depend strongly on Perelman's proof of the geometrization conjecture (see \cite{Per02, Per03, KleinerLott06,  MorganTian07, CaoZhu06, MorganTian08}). 

{\bf Acknowledgements:} We thank Chris Atkinson and Jeff Weeks for helpful suggestions.
We also thank Rupert Venzke  and Danny Calegari for a helpful discussion about the
minimal volume hyperbolic manifold with $n$ cusps and for 
correcting a mistaken conjecture in the first version of the paper.  We thank the referee for many helpful
comments, and for pointing out a strengthening of the main theorem. 

\section{Definitions}
In this section, we set up some notation and terminology for the theory of characteristic submanifolds, which is a relative version of the geometric decomposition. 

For a surface $X\subset M$, we'll use $M\split X$ to indicate 
the path-metric closure of $M\backslash X$. 
Let $M$ be an irreducible, orientable manifold. We will say that a properly embedded surface is {\it essential} if it is $\pi_1$-injective and $\partial \pi_1$-injective.
Let $(X, \partial X) \subset (M,\partial M)$ be a properly embedded essential surface. A {\it compressing annulus} for $X$ is an embedding 
$$ i: (S^1\times I, S^1\times\{0\} , S^1\times \{1\}) \hookrightarrow (M, X, \partial M)$$
such that 
\begin{itemize}
\item
$i_{\ast}$ is an injection on $\pi_1$,
\item 
$i(S^1\times I) \cap X = i(S^1\times \{0\})$, and 
\item
$i(S^1\times\{0\})$ is not \hypertarget{annular compression}{isotopic} in $X$ to $\partial X$. 
\end{itemize}

An {\it annular compression} of $(X,\partial X) \subset (M,\partial M)$ is a surgery 
along a compressing annulus $ i: (S^1\times I, S^1\times\{0\} , S^1\times \{1\}) \hookrightarrow (M, X, \partial M)$. 
Let $U$ be a regular neighborhood of $i(S^1\times I)$ in $M\split X$, let $\partial_1 U$ be the frontier of $U$ in
$M\split X$, and let $\partial_0 U=\partial U \cap (X\cup \partial M)$. Then let $X'= (X-\partial_0 U)\cup \partial_1 U$. 
The surface $X'$ is the \hyperlink{annular compression}{annular compression} of $X$ (see Figure \ref{annularcompression}). We remark that if $X$ is essential, then $X'$ is as well. 

\begin{figure}[htb] 
	\begin{center}
	\psfrag{m}{$\d M$}
	\psfrag{a}{$i(S^1\times I)$}
	\psfrag{F}{$X$}
	\psfrag{g}{$X'$}
	\psfrag{U}{$U$}
	\psfrag{v}{$\d_1 U$}
	\epsfig{file=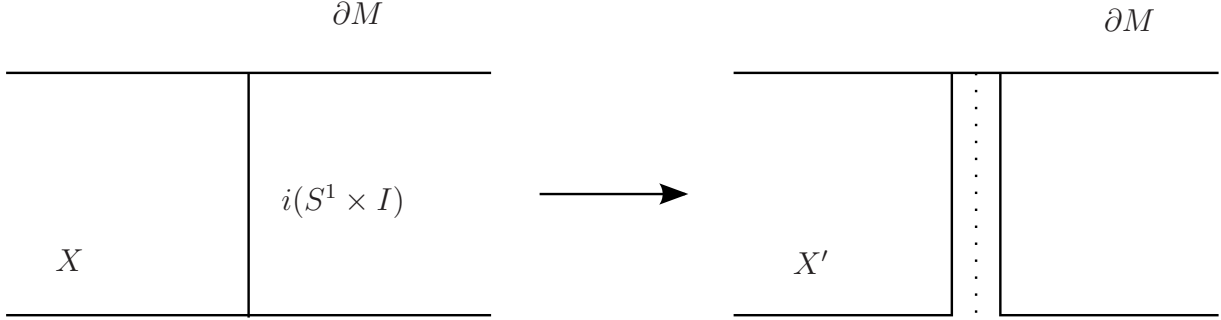, width=\textwidth}
	\caption{\label{annularcompression} An annular compression of a surface (cross the picture with $S^1$)}
	\end{center}
\end{figure}

The notion of a \hyperlink{pared}{pared manifold} was defined by Thurston to give a topological 
characterization of geometrically finite hyperbolic \hypertarget{pared}{3-manifolds} (see \cite{Mo} or \cite[Ch. 5]{CM04}). 
\begin{definition}
A {\it pared manifold} is a  pair $(M,P)$ where 
\begin{itemize}
\item
$M$ is a compact, orientable irreducible 3-manifold and
\item
$P\subset \partial M$ is a union of essential annuli
and tori in $M$,
\end{itemize}
such that
\begin{itemize}
\item
every abelian, noncyclic subgroup of $\pi_1(M)$ is peripheral
with respect to $P$ (i.e., conjugate to a subgroup of the fundamental
group of a component of $P$) and
\item 
every map $ \varphi:(S^1\times I, S^1\times \partial I) \to (M,P)$ that
is injective on the fundamental groups deforms, as maps of pairs, into $P$. 
\end{itemize}
$P$ is called the {\it parabolic locus} of the \hyperlink{pared}{pared manifold} $(M,P)$.
We denote by $\partial_0 M$ the surface $\partial M - int(P)$. 
\end{definition}

The motivation for the introduction of \hyperlink{pared}{pared manifolds} is the following geometrization theorem of Thurston:
\begin{theorem} \cite[Ch. 7]{CM04}
If $(M, P)$ is an oriented \hyperlink{pared}{pared 3-manifold}
with nonempty boundary, then there exists a geometrically finite uniformization of $(M, P)$.
\end{theorem}
The term uniformization means that there is a hyperbolic manifold $N$ such that $N$ is
homeomorphic to $int(M)$, and whose parabolic subgroups correspond to $P$ (see \cite[Ch. 7]{CM04} for more details).

Let $(M,P)$ be a \hyperlink{pared}{pared manifold} such that $\partial_0 M$ is incompressible.
There is a canonical set of essential annuli $(A,\partial A)\subset (M,\partial_0 M)$, called the {\it characteristic annuli}, 
such that $(P,\partial P)\subset (A,\partial A)$, and
characterized (up to isotopy) by the property that they are the maximal collection of non-parallel essential
annuli such that every other essential annulus $(B,\partial B)\subset (M,\partial_0 M)$
may be relatively isotoped to an annulus $(B',\partial B')\subset (M,\partial_0 M)$
so that $B'\cap A=\emptyset$.
Each complementary component  $L\subset M\split A$  is 
one of  the
following types:

\begin{enumerate}
\item
$T^2\times I$, a neighborhood of a torus component of $P$,
\item 
$(S^1\times D^2, S^1\times D^2 \cap \partial_0 M)$, a solid torus with annuli in the boundary,
\item
$(I-bundles, \partial I -subbundles)$, where the $I$-bundles over
the boundary are subsets of $A$, or
\item
all essential annuli in $(L,\partial_0 M\cap L)$ are parallel in $L$ into 
$(L\cap A, \partial(L\cap A) ).$
\end{enumerate}

The \hypertarget{window}{union} of components
of type $(3)$, denoted $(W,\partial_0 W)\subset (M, \partial_0 M)$, is called the {\it window} of $(M,\partial_0 M)$. 
It is unique up to isotopy of pairs. 

\hypertarget{guts}{The} \hyperlink{pared}{pared submanifold} $(M-W, \partial(M-W)- \partial_0 M)$ is denoted $Guts(M,P)$.  Note that the parabolic locus of $Guts(M,P)$ will consist of characteristic annuli. If $M$ is compact \hypertarget{surfaceguts}{orientable} and $int(M)$ admits a metric of finite volume, and $(X,\d X)\subset (M,\d M)$ is an essential surface,
then define $Guts(X)=Guts(M\split X,\d M\split \d X)$. 
The components of type (4) are acylindrical \hyperlink{pared}{pared manifolds}, which have a complete hyperbolic
structure of finite volume with geodesic boundary \cite{Mo}. 
We will let $\Vol(Guts(M,P))$ denote the volume of this hyperbolic metric. If $D(M,P)$ is
obtained by taking two copies of $M$ and gluing them along the corresponding surfaces
$\d_0 M$, then $\Vol(Guts(M,P))=\frac12 \Vol( D(M,P) )$, where $\Vol(D(M,P))$ is
the simplicial volume of $D(M,P)$, {\it i.e.} the sum of the volumes of the hyperbolic
pieces of the geometric decomposition. If $\Vol(Guts(M,P))=0$, then $M$ is a ``book of $I$-bundles",
with ``pages" consisting of $W$, and ``spine" consisting of solid tori and $T^2\times I$  (see \cite[Example 2.10.4]{CM04} for more information). 

Let $\hypertarget{V8}{V_8}=4\cdot K = 3.66\ldots$, where $K$ is Catalan's constant
$$K= 1-1/9 + 1/25 -1/49 + \cdots + (-1)^n/ (2n+1)^2+ \cdots.$$ 
Then $V_8$ is also the volume of a regular ideal octahedron in $\HH^3$.

\section{Essential surfaces}
In this section, we prove the existence of an essential surface with non-trivial \hyperlink{surfaceguts}{guts} under certain hypotheses (Theorem \ref{gutsy surface}).
In the terminology of Culler and Shalen, we find a surface which is not a {\it fibroid} \cite{CS1}, {\it i.e.} whose complement is not a book of $I$-bundles. 
We start with an essential surface missing one cusp. If this surface has no accidental parabolics, we are done. If it does, we perform annular
compressions until we get a surface with no accidental parabolics and which is not a fibroid. 

\begin{lemma} \label{compression}
Let $(M,P)$ be a \hyperlink{pared}{pared manifold}, with \hyperlink{window}{window}
$(W,\d_0 W)\subset (M, \partial_0 M)$. Suppose that there is a component
$(J, K)\subset (W, \d_0 W)$ which is an $I$-bundle such that $\chi(J)<0$,
and such that $K\cap P\neq \emptyset$. Let $K_0$ be a component of $K\cap P$. 
Then the surface $\partial_0 M\cup K_0$ is compressible in $M$. 
\end{lemma}
\begin{proof}
Let $Q$ be a surface such that $I\to J\to Q$ is a fiber bundle. Then 
$I\to K\to \partial Q$ is a bundle over the boundary. Let $q_0\subset \partial Q$
be the component of $\partial Q$ such that $K_0$ fibers over $q_0$. 
Let $(\alpha,\partial \alpha) \subset (Q, q_0)$ be an essential arc. 
Then there is a disk $D\subset J$ which fibers over $\alpha$. 
The disk $D$ is essential, which implies that $\d_0 M\cup K_0$ is compressible. 
To see this, note that if $D$ is non-separating in $M$, then it is 
essential. So we may assume that $D$ separates $M$, and therefore
$D$ separates $J$. But $J$ is $\pi_1$-injective in $M$, and
since $\alpha$ is essential in $Q$, $D$ is essential in $J$. Thus,
each component of $J-D$ is $\pi_1$-injective in $J$, and therefore
in $M$, and neither component of $J-D$ is simply-connected. Therefore
neither component of $\overline{M-D}$ can be a ball, since it contains an non-simply
connected $\pi_1$-injective submanifold, which implies that $D$ is essential,
and therefore that $\partial_0 M \cup K_0$ is compressible. 
\end{proof}

\begin{lemma} \label{subpared}
Let $M$ be an orientable compact manifold such that $int(M)$ is hyperbolic of finite volume and
so that $P=\d M$ is the pared locus of $M$. Let $X\subset M$ be an essential surface. 
Then $(M\split X, P\split \d X)$ is a \hyperlink{pared}{pared manifold}. 
\end{lemma}
\begin{proof}
Since each component $V$ of  $M\split X$ is $\pi_1$-injective, there is a covering space $\tilde{M}\to M$ such
that there is a lift $M\split X\to \tilde{M}$ which is a homotopy equivalence. Any
abelian non-cyclic subgroup of $\pi_1(V)$ must be peripheral in $M$ and
therefore in $\tilde{M}$ and $M\split X$. Any  map $ \varphi:(S^1\times I, S^1\times \partial I) \to (M\split X,P\split \partial X)$ that is injective on the fundamental groups deforms, as maps of pairs, into $P\split X\subset \tilde{M}$,
since components of $P$ correspond to cusps associated to a complete hyperbolic structure on $int(\tilde{M})$ induced
from the hyperbolic structure on $int(M)$. Thus, we see that $(M\split X, P\split \partial X)$ satisfies
the hypotheses of a \hyperlink{pared manifold}{pared manifold}.
\end{proof}

\begin{lemma} \label{paredgut}
Let $M$ be an orientable compact manifold such that $int(M)$ is hyperbolic of finite volume,
so that $P=\d M$ is the pared locus of $M$. 
Let $X\subset M$ be an essential surface. If $X$ 
has a compressing annulus, let $X_1$ be the
surface obtained by performing an \hyperlink{annular compression}{annular compression}
along $X$. Then $X_1$ has a pared annulus coming from $\d M$ in the
boundary of one of its gut regions.  
\end{lemma}

\begin{proof}
Let $A \subset M\split X$ be a compressing annulus, with 
$\partial A = a_0 \cup a_1$, such that $a_0 \subset \partial M$,
and $a_1\subset X$ is a closed curve in $X$ which is not
boundary parallel in $X$. By Lemma \ref{subpared}, $(M\split X,P\split \partial X)$
is a \hyperlink{pared}{pared manifold}. Compression along $A$ creates the essential
surface $X_1$ (see Figure \ref{annularcompression}). There will be a new component of the
pared locus of $M\split X_1$ which is an annulus with core
$a_0$. This annulus must be in a component of \hyperlink{surfaceguts}{$Guts(X_1)$},
because otherwise $a_0$ would lie in the boundary of the
core of the \hyperlink{window}{window} $W\subset M\split X_1$.
This implies that $X$ had a compression,
since reversing the \hyperlink{annular compression}{annular compression} to obtain $X$ from $X_1$ corresponds to adding a boundary annulus to an $I$-bundle of Euler characteristic $<0$, creating
a compression by Lemma \ref{compression}. 
\end{proof}

\begin{theorem} \label{gutsy surface}
Let $M$ be an orientable compact  3-manifold whose interior admits a hyperbolic metric of finite volume.
Suppose that $\partial M$ contains a torus $T_1$, and that there is a 2-sided essential surface
$X_0 \subset M$ such that $\partial X_0 \cap T_1=\emptyset$.  Then there
is an essential surface $X \subset M$ such that $\chi(Guts(X))<0$. 
\end{theorem}

\begin{proof}
Let $\partial M = T_{1}\cup \cdots \cup T_{n}$, where $T_{i}$ is a torus. Then $\partial X_0\subset T_2\cup \cdots \cup T_n$. 
Suppose that $\chi(Guts(X_0))=0$. Then $M\split X_{0}$ is a book of $I$-bundles. 
Now, we perform a maximal sequence of \hyperlink{annular compression}{annular compressions} along $\partial M$ to obtain
a sequence of essential surfaces $X_{1}, X_{2}, \ldots X_{k} \subset M$, where $X_{i+1}$ is obtained from $X_{i}$
by an \hyperlink{annular compression}{annular compression}. Since $X_0$ is disjoint from $T_1$, the component of \hyperlink{surfaceguts}{$Guts(X_0)$} incident with $T_1$ must be of the form $T^2\times I$, and thus $X_0$ has an \hyperlink{annular compression}{annular compression}. The compressing annulus comes from a curve in $T^2$ which is parallel to the boundary of a characteristic annulus of $M\split X_0$, and which is parallel to a curve in $T_1$ via a compressing annulus.  For $i>0$, $X_{i}$ will be a 2-sided essential surface,
with a component of the pared locus coming from $\d M$ lying in the boundary
of a gut region of $M\split X_i$, by Lemma \ref{paredgut}. If $\chi(Guts(X_i))=0$ for $i>0$,
then the components of \hyperlink{surfaceguts}{$Guts(X_i)$} must be solid tori or tori $\times I$. Since $M$ is simple,
a pared annulus in $\d M$ intersect the boundary of a solid torus  component of \hyperlink{surfaceguts}{$Guts(X_i)$} must be primitive,
and thus there is an \hyperlink{annular compression}{annular compression} of $X_i$ (see Figure \ref{induction}). Any region of \hyperlink{surfaceguts}{$Guts(X_i)$} of the form $T^2\times I$ will also give rise to an \hyperlink{annular compression}{annular compression} of $X_i$, as we observed for $X_0$. 
Since $\chi(X_{i+1})=\chi(X_i)$ and $X_{i+1}$ has two
more boundary components than $X_i$,
the number of \hyperlink{annular compression}{annular compressions} is finite, and thus at
some stage we arrive at $X_k$ such that $\chi(Guts(X_k))<0$. Let $X=X_k$. 
\end{proof}

\begin{figure}[htb] 
	\begin{center}
	\psfrag{G}{\hyperlink{surfaceguts}{$Guts(X_i)$}}
	\psfrag{a}{$a$}
	\psfrag{Y}{$X_{i+1}$}
	\psfrag{X}{$X_i$}
	\psfrag{M}{$\d M$}
	\psfrag{s}{$W$}
	\psfrag{H}{\hyperlink{surfaceguts}{$Guts(X_{i+1})$}}
	\epsfig{file=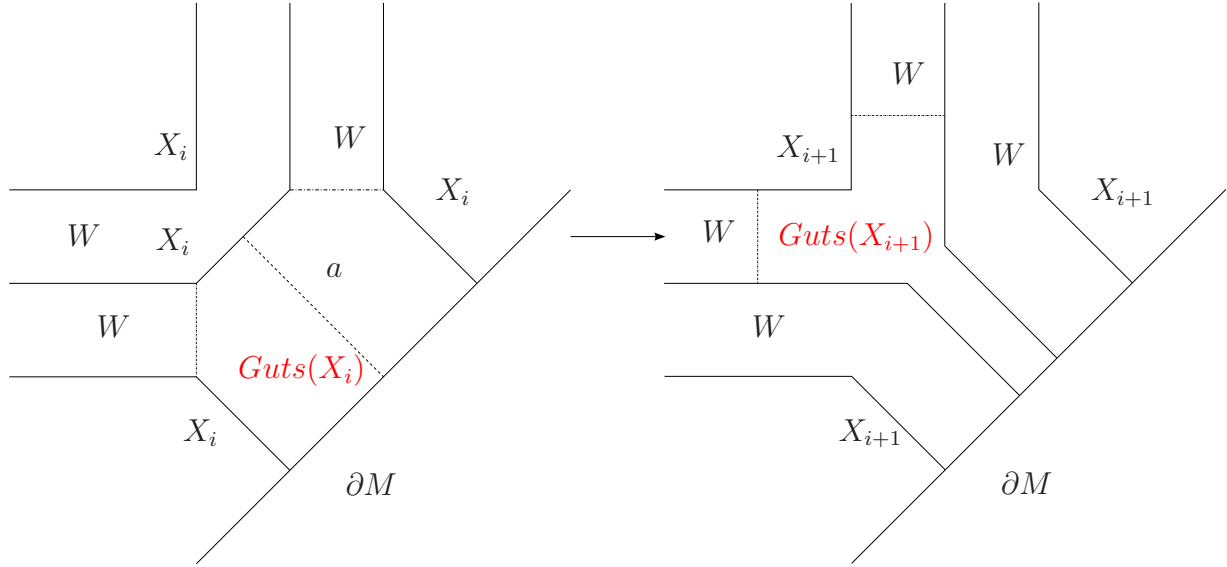, width=\textwidth}
	\caption{\label{induction} An annular compression $a$ of $X_i$ gives $X_{i+1}$ (cross the picture with $S^1$)}
	\end{center}
\end{figure} 

\begin{theorem} \label{surface volume}
Let $M$ be an orientable compact  3-manifold whose interior admits a hyperbolic metric of finite volume.
Suppose that $\partial M$ contains a torus $T_1$, and that there is a 2-sided essential surface
$X_0 \subset M$ such that $\partial X_0 \cap T_1=\emptyset$.  Then $\vol(M)\geq V_8$. If $\vol(M)=V_8$,
then $M$ has a decomposition into a single right-angled octahedron.  
\end{theorem}
\begin{proof}
From Theorem \ref{gutsy surface}, there is an essential 
surface $X\subset M$ such that $\chi(Guts(X))\leq -1$. By
\cite[Theorem 5.5]{CFW08}, we have $\Vol(M)\geq \Vol(Guts(X))$ (see also \cite[Theorem 9.1]{AST05} on which the method of proof
of \cite{CFW08} is based). 
By \cite[Theorem 4.2]{Mi}, $\Vol(Guts(X))\geq -V_8 \chi(Guts(X))$.
Thus, we have $\Vol(M)\geq V_8$. 

If $\Vol(M)=V_8$, then
$\Vol(M)=\Vol(Guts(X))$. We may assume that no complementary
region of $M\split X$ is an $I$-bundle, since otherwise we could
replace $X$ with a surface which it double covers.  By \cite[Theorem 5.5]{CFW08},
$X$ is totally geodesic, and so $N=M \split X = Guts(X)$. Regardless of whether
$X$ is one- or two-sided, we have $\chi(X) = \frac12 \chi(\partial N) = \chi(N) = -1$,
so $X$ is either a punctured torus, a thrice-punctured sphere, or a twice-puncture $\RR\PP^2$. 
Since $N$ is an acylindrical manifold with 
$\chi(N)=-1$ and $\Vol(N)=V_8$, by Miyamoto's theorem \cite[Theorem 4.2]{Mi}
$N$ must have a decomposition into a single regular octahedron
of volume $V_8$. Thus $M$ decomposes into a single regular octahedron.  
\end{proof}

{\bf Remark:} The  referee  indicated the previous theorem as a natural 
generalization of our main result, and pointed out that there is a sharp example (see \cite[Section 6]{Dunfield99}). 
The example is m137 from the Snappea census \cite{Snappea}, and is obtained
by $0$-framed surgery on a 2-component link $L$ (see Figure \ref{m137}). To see that m137
contains a closed essential surface, note that there is a thrice-punctured sphere $P$ in the
complement of m137 (see the grey surface in the second Figure \ref{m137}). The
surgery along the boundary slope of $P$ gives $S^2\times S^1$, since the other component of $L$ is 
unknotted. However, since m137 is hyperbolic,
it does not fiber over $S^1$ with fiber $P$, and therefore by \cite[Theorem 2.0.3]{CGLS}, m137 must
contain a closed essential surface (which is obtained by carefully tubing together the boundary
components of two copies of $P$). 

\begin{figure}[htb]
  \epsfig{figure=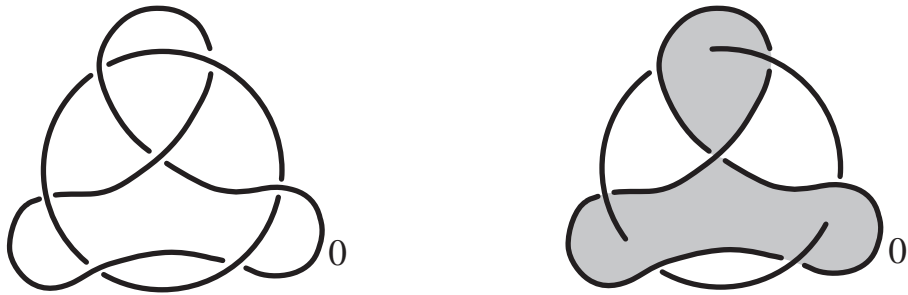,angle=0}
        \caption{\label{m137} m137 realizes the lower bound in Theorem \ref{surface volume}, and a thrice-punctured sphere $P$ in m137}
     
\end{figure}

\begin{theorem} \label{two cusps}
Suppose that $M$ is an orientable hyperbolic 3-manifold with two cusps, then 
$\vol(M)\geq V_8=3.66...$. 
If $\vol(M) = V_8$, then $M\cong \mathbb{W}$ or $M\cong \mathbb{W}'$.
\end{theorem}
\begin{proof}
By \cite[Theorem 3]{CullerShalen84},
there is an orientable connected essential separating surface $X_{0}\subset M$ such that $\partial X_{0}\cap (T_{2}\cup \cdots \cup T_n)=\emptyset$,
and $\partial X_{0} \cap T_{1}$ is a non-empty union of non-contractible simple closed curves. Moreover, we may assume that
there is no compressing annulus of $X_0$ along $T_1$. By Theorem \ref{surface volume}, $\Vol(M)\geq V_8$. 

Now, suppose $\Vol(M)=V_8$. By Theorem \ref{surface volume}, $M$ contains an essential surface $X$ with $\chi(X)=-1$, and the complement of $X$ has a decomposition into a regular octahedron. 

\begin{figure}[htb]
  \subfigure{\epsfig{figure=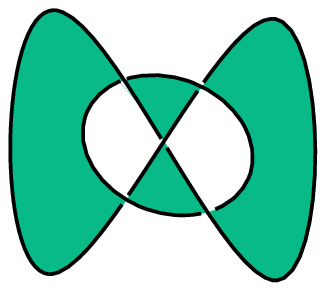,angle=0,width=.45\textwidth}}\qquad
    \subfigure{\epsfig{figure=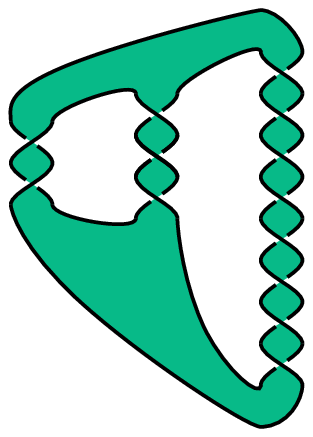,angle=0,width=.45\textwidth}}
        \caption{\label{Whitehead} A twice-punctured $\RR\PP^2$
        in the Whitehead and $(-2,3,8)$-pretzel link complements}
\end{figure}

To finish the argument, we take an octahedron $O$ and color
the faces alternately black and white. The argument here is
somewhat indirect. We know we have two examples, $\mathbb{W}, \mathbb{W}'$,
and we need to show that there are no more. We glue the black
faces of $O$ together in all possible ways to get the possible
manifolds $N$. Up to isometry, we get four manifolds $N_i, i=1,2,3,4$ with
totally geodesic boundary and rank one cusps (see Figure \ref{octahedra}, where
the black faces being paired are labelled $A$ and $B$).
\begin{figure}[htb] 
	\begin{center}
	\psfrag{A}{$A$}
	\psfrag{B}{$B$}
	\psfrag{a}{$N_1$}
	\psfrag{b}{$N_2$}
	\psfrag{c}{$N_3$}
	\psfrag{d}{$N_4$}
	\epsfig{file=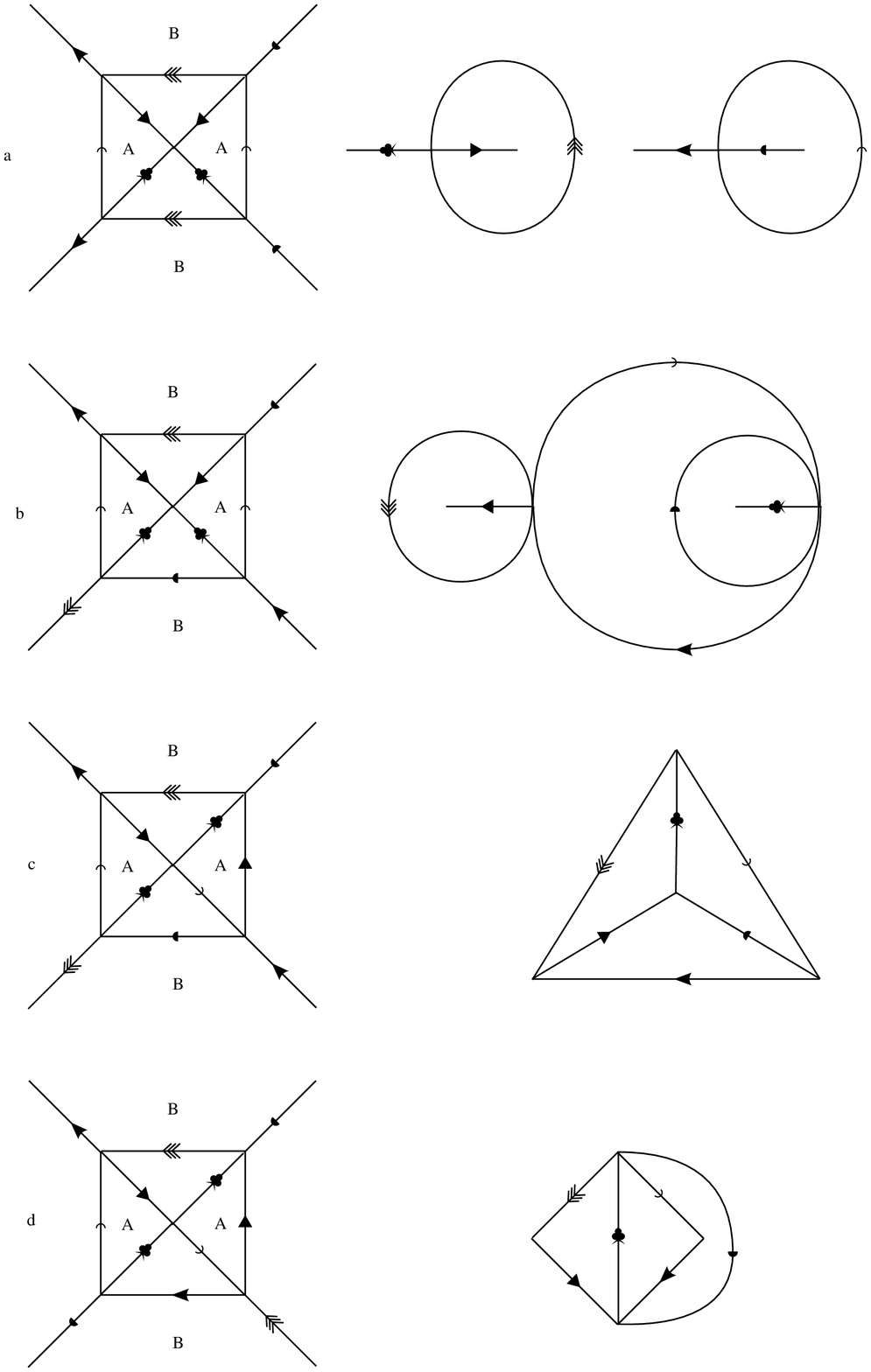, width=.75\textwidth}
	\caption{\label{octahedra} Gluing faces $A$  and $B$ of an octahedron together in pairs  to get manifolds with geodesic
	boundary of volume $V_8$}
	\end{center}
\end{figure} 
Then $\d N_1$ is two pairs of  pants, and $\d N_i$ is a single
four-punctured sphere, for $i=2,3,4$ ($\d N_i$ is shown on the
right of Figure \ref{octahedra}). Gluing the two boundary
components of $N_1$ together in a way that gives two cusps 
gives $\mathbb{W}$ the Whitehead link complement. There are
three other manifolds which have boundary a four-punctured
sphere. The boundary gets glued to itself by an antipodal
map. To get a manifold with two cusps from $N_i$, the antipodal map
must also identify the rank one cusps of $\d N_i$ which share
a common rank one cusp of $N_i$. 
  Only $N_2$ and $N_4$ have boundary which
admits an antipodal isometry, so that the quotient is 
$X\cong \RR\PP^2-\{x,y\}$. To see that $\d N_3$ cannot have
such an antipodal map, notice that $\d N_3$ has isometry group
$S_4$, for which there is no isometry which acts as an antipodal map. The antipodal isometry preserves
the induced triangulation of $\d N_i$ shown on the right of 
Figure \ref{octahedra} for $i=2,4$. One may check that the quotient by these isometries
give a manifold with two cusps. Any other isometry which had
this property must be the same, since the product of two such
isometries must fix all four cusps of $\d N_i$, which implies
that it is the identity since this is the only holomorphic map
of $S^2$ fixing four points. One quotient by this map gives
back $\mathbb{W}$, since $\mathbb{W}$ has a twice-punctured $\RR\PP^2$
in its complement (see Figure \ref{Whitehead}).
 The other example must give the $(-2,3,8)$-pretzel link complement $\mathbb{W}'$ (in fact, $N_2$
 corresponds to $\mathbb{W}'$ and $N_4$ corresponds to $\mathbb{W}$, after one quotients by the unique
 antipodal map of the boundaries).  
\end{proof}

\section{Conclusion}

It would be interesting to find the minimal volume orientable hyperbolic
manifolds with $n$ cusps. We conjecture that the
minimal volume $n$-cusped manifold is realized by a hyperbolic chain link which is minimally
``twisted'' (see \cite{NR92} for the definition of chain links) for $n\leq 10$.
Rupert Venzke has pointed out to us that for links with $n$ cusps,
where $n\geq 11$, the $n-1$-fold cyclic cover $\mathbb{W}_n$ over one component of $\mathbb{W}$ has volume less
than the smallest volume twist link with $n$ cusps, which is a
Dehn filling on $\mathbb{W}_{n+1}$. 
 It is probably impossible to use the methods in this
paper to determine the smallest volume hyperbolic manifold with 
$n$ cusps. However, it should be possible to 
prove the following:

\begin{conjecture}
Let $v_n$ be the minimal volume of an $n$-cusped orientable hyperbolic 3-manifold. Then 
$$\underset{n\to\infty}{\lim} v_n/n = V_8.$$
\end{conjecture}
As mentioned before, Adams has shown that $v_n/n \geq V_3$ \cite{Adams88}. 
It may also be possible to prove:

\begin{conjecture}
The minimal limit volume of non-orientable hyperbolic 3-manifolds
is $V_8$.
\end{conjecture}

In the Snappea census of 3-manifolds \cite{Snappea}, there
are four manifolds of volume $V_8$ which are non-orientable
with a single orientable cusp, which would correspond to
the smallest limit volumes of non-orientable manifolds if this conjecture were true. It seems possible that one could prove this conjecture
by proving the existence of an essential surface disjoint from the orientable cusp with non-trivial
\hyperlink{surfaceguts}{guts}. 

\bibliographystyle{hamsplain}
\bibliography{twocusps_revision.bbl}

\end{document}